\documentclass[11pt]{article}
\usepackage{amscd,amssymb,amsmath,verbatim,amsthm}
\usepackage{color}
\usepackage{comment}
\date{}

\newtheorem{thm}{Theorem}[section]
\newtheorem{lemma}[thm]{Lemma}
\newtheorem{prop}[thm]{Proposition}

\newtheorem{remark}{Remark}

\newcommand{\lp}{\left(}
\newcommand{\rp}{\right)}

\newcommand{\dist}{\text{dist}} %distance
\newcommand{\divergence}{\text{div}}

\newcommand{\abs}[1]{\left\lvert #1\right\rvert}

\newcommand{\be}{\begin{equation}}
\newcommand{\ee}{\end{equation}}
\newcommand{\bee}{\begin{equation*}}
\newcommand{\eee}{\end{equation*}}
\newcommand{\bea}{\begin{eqnarray}}
\newcommand{\eea}{\end{eqnarray}}
\newcommand{\bs}{\begin{split}}
\newcommand{\es}{\end{split}}

\newcommand{\CC}{\mathbb C}
\newcommand{\HHs}{\mathbb H}

\newcommand{\RR}{\mathbb R}

\newcommand{\del}{\partial}

\newcommand{\e}{\varepsilon}

\newcommand{\Ric}{\mathrm{Ric}}

\newcommand{\calC}{{\mathcal C}}

\newcommand{\calO}{{\mathcal O}}

\newcommand{\calT}{{\mathcal T}}
\newcommand{\calU}{{\mathcal U}}

\newcommand{\olg}{\overline{g}}

\begin{document}

\title{Singular solutions of \\ fractional order conformal Laplacians}
\author{Maria del Mar Gonz\'alez \\ Universitat Politecnica de Catalunya
\and Rafe Mazzeo \\ Stanford University \and Yannick Sire \\Universit\'e Paul C\'ezanne, Aix-Marseille III}

\maketitle

\begin{abstract}
We investigate the singular sets of solutions of a natural family of conformally covariant pseudodifferential elliptic operators
of fractional order, with the goal of developing generalizations of some well-known properties of solutions of the singular
Yamabe problem.
\end{abstract}
%\tableofcontents
% -------------------------------------------------------------------------------------------------------
% -------------------------------------------------------------------------------------------------------

%\tableofcontents

\section{Introduction}

Let $(M,\olg)$ be a compact $n$-dimensional Riemannian manifold, $n \geq 3$. If $\Lambda \subset M$ is any
closed set, then the `standard' singular Yamabe problem concerns the existence and geometric properties of
complete metrics of the form $g = u^{\frac{4}{n-2}}\olg$ with constant scalar curvature. This corresponds to solving
the partial differential equation
\begin{equation}\label{eq:cL}
\Delta_{\olg} u + \frac{n-2}{4(n-1)} R^{\olg} u = \frac{n-2}{4(n-1)}R^g \, u^{\frac{n+2}{n-2} }, \qquad u > 0,
\end{equation}
where $R^g$ is constant and with a `boundary condition' that $u \to \infty$ sufficiently quickly at $\Lambda$ so that
$g$ is complete. (Note that in our convention, $\Delta_{\olg}$ is an operator with nonnegative spectrum.) It is known
that solutions with $R^g < 0$ exist quite generally if $\Lambda$ is large in a capacitary
sense \cite{Lab}, whereas for $R^g > 0$ existence is only known when $\Lambda$ is a smooth
submanifold (possibly with boundary) of dimension $k < (n-2)/2$, see \cite{MP}, \cite{F}.

There are both analytic and geometric motivations for studying this problem. For example, in the positive case ($R^g > 0$), solutions
to this problem are actually weak solutions across the singular set, so these results fit into the broader investigation of
possible singular sets of weak solutions of semilinear elliptic equations. On the geometric side is a well-known theorem
by Schoen and Yau \cite{SY} stating that if $(M,h)$ is a compact manifold with a locally conformally flat metric $h$ of
positive scalar curvature, then the developing map $D$ from the universal cover $\widetilde{M}$ to $S^n$, which
by definition is conformal, is injective, and moreover, $\Lambda := S^n \setminus D(\widetilde{M})$ has Hausdorff
dimension less than or equal to $(n-2)/2$. Regarding the lifted metric $\tilde{h}$ on $\widetilde{M}$ as a metric
on $\Omega$, this provides an interesting class of solutions of the singular Yamabe problem which are periodic with respect to a
Kleinian group, and for which the singular set $\Lambda$ is typically nonrectifiable. More generally, that paper also
shows that if $\olg$ is the standard round metric on the sphere and if $g = u^{\frac{4}{n-2}}\olg$ is a complete metric
with positive scalar curvature and bounded Ricci curvature on a domain $\Omega = S^n \setminus \Lambda$, then
$\dim \Lambda \leq (n-2)/2$.

In the past two decades it has been realized that the conformal Laplacian, which is the operator appearing as the
linear part of \eqref{eq:cL}, fits into a holomorphic family of conformally covariant elliptic pseudodifferential operators.
The operators in this family of positive even integer order are the GJMS operators, and these have a central role in conformal
geometry. Just as the Yamabe problem is naturally associated to the conformal Laplacian, so too are there higher order
Yamabe-type problems associated to the other GJMS operators, or more generally, also to these other conformally covariant
operators with noninteger order. The higher (integer) order Yamabe problems have proved to be analytically challenging
and provide insight into the GJMS operators themselves. Hence it is reasonable to hope that these fractional order (singular)
Yamabe problems will have a similarly rich development and will bring out interesting features of these
conformally covariant pseudodifferential operators. From a purely analytic perspective, little is known about regularity of
solutions of semilinear pseudodifferential
equations like these, and this family of geometric problems is a natural place to start. In fact, as we explain below,
fractional powers of the Laplacian have also appeared recently in the work of Caffarelli and his collaborators as
generalized Dirichlet to Neumann operators for certain singular divergence form elliptic equations, which further
indicates the worth of studying such operators.

The present paper begins an investigation into these questions. Our goals here are limited: beyond presenting
this set of problems as an interesting area of investigation, we prove a few
results which indicate how certain properties of the fractional singular Yamabe problem extend some well-known
results for the standard Yamabe equation.

To describe this more carefully, we first define the family of fractional conformal powers of the Laplacian. As we have
already indicated, the linear operator which appears as the first two terms on the left in \eqref{eq:cL} is known as
the conformal Laplacian associated to the metric $\olg$, and denoted $P_1^{\olg}$. It is conformally covariant in the sense
that if $f$ is any (smooth) function and $g = u^{\frac{4}{n-2}}\, \olg$ for some $u > 0$, then
\begin{equation}
P_1^{\olg} (uf) = u^{\frac{n+2}{n-2}}P_1^g (f).
\label{eq:cccL}
\end{equation}
Setting $f \equiv 1$ in \eqref{eq:cccL} yields the familiar relationship \eqref{eq:cL} between the scalar curvatures
$R^{\olg}$ and $R^g$.  $P_1$ is the first in a sequence of conformally covariant elliptic operators, $P_k$,
which exist for all $k \in {\mathbb N}$ if $n$ is odd, but only for $k \in \{1, \ldots, n/2\}$ if $k$ is even. The first construction
of these operators, by Graham-Jenne-Mason-Sparling \cite{GJMS} (for which reason they are known as the GJMS operators),
proceeded by trying to find lower order geometric correction terms to $\Delta^k$ in order to obtain nice transformation
properties under conformal changes of metric. Beyond the case $k=1$ which we have already discussed, the operator
\[
P_2=\Delta^2 +\delta\lp a_n Rg+b_n Ric\rp d+\tfrac{n-4}{2}Q_2,
\]
called the Paneitz operator (here $Q_2$ is the standard $Q$-curvature), had also been discovered much earlier than the operators $P_k$ with $k > 2$.

This leads naturally to the question whether there exist any conformally covariant pseudodifferential operators of noninteger
order. A partial result in this direction was given by Peterson \cite{P}, who showed that for any $\gamma$, the
conformal covariance condition determines the full Riemannian symbol of a pseudodifferential operator with
principal symbol $|\xi|^{2\gamma}$. Hence $P_\gamma$ is determined modulo smoothing operators, but it is by no
means clear that one can choose smoothing operators to make the conformal covariance relationships hold exactly.
The breakthrough result, by Graham and Zworski \cite{GZ}, was that if $(M,[\olg])$ is a smooth compact manifold
endowed with a conformal structure, then the operators $P_k$ can be realized as residues at the values $\gamma = k$ of
the meromorphic family $S(n/2 + \gamma)$ of scattering operators associated to the Laplacian on any Poincar\'e-Einstein manifold
$(X,G)$ for which $(M,[\olg])$ is the conformal infinity.   These are the `trivial' poles of the scattering operator, so-called
because their location is independent of the interior geometry; $S(s)$ typically has infinitely many other poles, which are
called resonances, the location and asymptotic distribution of which is a matter of considerable interest and ongoing study.
Multiplying this scattering family by some $\Gamma$ factors to regularize these poles, one obtains a holomorphic family of
elliptic pseudodifferential operators $P_\gamma^{\olg}$ (which patently depends on the filling $(X,G)$). An alternate construction of
these operators has been obtained by Juhl, and his monograph \cite{Juhl} describes an intriguing general framework for
studying conformally covariant operators, see also \cite{Juhl2}.

This realization of the GJMS operators has led to important new understanding of them, including for example the basic fact that
$P_\gamma^{\olg}$ is symmetric with respect to $dV_{\olg}$, (something not obvious from the previous fundamentally algebraic
construction). Hence even though the family $P_\gamma^{\bar g}$ is not entirely canonically associated to $(M,[\olg])$ (as we
explain in some detail below), its study can still illuminate the truly canonical operators which occur as special values at
positive integers, i.e.\ the GJMS operators.

For various technical reasons, we focus here only on the operators $P_\gamma$ when $\gamma \in \RR$, $|\gamma| \le n/2$.
These have the following properties: first, $P_0 = \mathrm{Id}$, and more generally, $P_k$ is the $k^{\mathrm{th}}$ GJMS
operator, $k = 1, \ldots, n/2$;  next, $P_\gamma$ is a classical elliptic pseudodifferential operator of order $2\gamma$
with principal symbol $\sigma_{2\gamma}(P_\gamma^{\olg}) = |\xi|^{2\gamma}_{\olg}$, hence (since $M$ is compact), $P_\gamma$
is Fredholm on $L^2$ when $\gamma > 0$; if $P_\gamma$ is invertible, then $P_{-\gamma} = P_{\gamma}^{-1}$; finally,
\begin{equation}
\mbox{if}\ g=u^{\frac{4}{n-2\gamma}}\olg, \qquad \mbox{then}\ P_\gamma^{\olg} (uf) = u^{\frac{n+2\gamma}{n-2\gamma}} P_\gamma^g (f)
\label{eq:ccfcL}
\end{equation}
for any smooth function $f$. Generalizing the formul\ae\ for scalar curvature ($\gamma = 1$) and the Paneitz-Branson
$Q$-curvature ($\gamma = 2$), we make the definition that for any $0 < \gamma \leq n/2$, $Q_\gamma^{\olg}$, the
$Q$-curvature of order $\gamma$ associated to a metric $\olg$, is given by
\begin{equation}
Q_\gamma^{\olg} = P_\gamma^{\olg}(1).
\label{eq:Qgamma}
\end{equation}

Let us comment further on the choices involved in these definitions. First, Poincar\'e-Einstein fillings $(X,G)$ of $(M,[\olg])$ (which
are defined at the beginning of \S 2), may
not always exist, and when they exist, they may not be unique.  The existence issue is not serious: the construction of \cite{GZ}
only uses that the metric $G$ satisfy the Einstein equation to sufficiently high order, and one can even take $X = M \times [0,1]$
with the conformal structure $[\olg]$ at $M \times \{0\}$ and with the other boundary $M \times \{1\}$ a regular (incomplete)
boundary for $G$. However, these comments indicate that the issue of lack of uniqueness is far worse, since there are
always infinite dimensional families of asymptotically Poincar\'e-Einstein fillings. Any choice of one of these fixes a
family of operators $P_\gamma^{\olg}$, and for each such choice $P_\gamma$ satisfies all the properties listed above.
As already noted, the complete Riemannian symbol of $P_\gamma^{\olg}$ is determined by the metric $\olg$ and
the conformal covariance; the choice of filling provides a consistent selection of smoothing terms in these
pseudodifferential operators for which the same covariance properties hold. Hence the $Q$-curvatures $Q_\gamma^{\olg}$ for
noninteger values of $\gamma$ are similarly ill-defined. In particular, except in certain special cases where there
are canonical choices of fillings (e.g.\ the sphere), it is not clear that the existence of a metric $\olg$ in a conformal
class such that $Q_\gamma^{\olg} > 0$ depends only on that conformal class.  We leave open these significant problems,
and in what follows, always make the tacit assumption that for any given $(M,[\olg])$, we have fixed an approximately
Poincar\'e-Einstein filling $(X,G)$ and used this to define the family $P_\gamma^{\olg}$. In other words, it is perhaps
more sensible to think of $P_\gamma$ and $Q_\gamma$ as quantities determined by the pair $(( M,[\olg]), (X,G))$.

In any case, generalizing \eqref{eq:cL}, consider the ``fractional Yamabe problem": given a metric $\olg$  on a compact
manifold $M$, find $u > 0$ so that if $g = u^{4/(n-2\gamma)}\olg$, then $Q_\gamma^g$ is constant. This amounts to solving
\begin{equation}
P_\gamma^{\olg} u = Q_\gamma^{g} u^{\frac{n+2\gamma}{n-2\gamma}} ,\quad u>0,
\label{eq:ccQ}
\end{equation}
for $Q_\gamma^g = \mbox{const.}$  More generally, we can simply seek metrics $g$ which are conformally related to $\olg$
and such that $Q_\gamma^g \geq 0$ or $Q_\gamma^g < 0$ everywhere.

This fractional Yamabe problem has now been solved in many cases where the positive mass theorem is not needed
\cite{fractional-Yamabe}, and further work on this is in progress.

As described earlier, it is is also interesting to construct complete metrics of constant (positive) $Q_\gamma$ curvature on open
subdomains $\Omega = M \setminus \Lambda$, or in other words, to find metrics $g = u^{4/(n-2\gamma)}\olg$ which are complete
on $\Omega$ and such that $u$ satisfies \eqref{eq:ccQ} with $Q_\gamma^g$ a constant.  This is the fractional singular Yamabe
problem. In the first few integer cases it is known that the positivity of the curvature places restrictions on $\dim \Lambda$ (see \cite{SY},
\cite{MP} for the case $\gamma=1$, \cite{Chang-Hang-Yang} for $\gamma=2$, and \cite{non-removable} for the analogous problem for the closely related $\sigma_k$ curvature).

Although it is not at all clear how to define $P^g_\gamma$ and $Q^g_\gamma$ on a general complete open manifold,
we can give a reasonable definition when $\Omega$ is an open dense set in a compact manifold $M$ and the metric
$g$ is conformally related to a smooth metric $\olg$ on $M$. Namely, we can define them by demanding that the
relationship \eqref{eq:ccfcL} holds.  Note, however, that this too is not as simple as it first appears since, because of
the nonlocal character of $P_\gamma^{\olg}$, we must extend $u$ as a distribution on all of $M$.
We discuss this further below.

%Provided we do so, it is straightforward to check that the
%resulting operators are essentially self-adjoint on $L^2(\Omega, dV_g)$ when $\gamma$ is real.
%It is also essentially self-adjoint then, as follows from the fact that $P_\gamma^g$ is a relatively compact perturbation of $(\Delta_g)^{\gamma}$.

The purpose of this note is to clarify some basic features of this fractional singular Yamabe problem and to establish
a few preliminary results about it. Our first result generalizes the Schoen-Yau theorem.
\begin{thm}\label{th:SY}
Suppose that $(M^n,\olg)$ is compact and $g = u^{\frac{4}{n-2\gamma}}\olg$ is a complete metric on $\Omega = M \setminus
\Lambda$,  where $\Lambda$ is a smooth $k$-dimensional submanifold.  Assume furthermore that $u$ is polyhomogeneous
along $\Lambda$ with leading exponent $-n/2 + \gamma$. If $0 < \gamma \leq \frac{n}{2}$, and if $Q_\gamma^g > 0$ everywhere
for any choice of asymptotically Poincar\'e-Einstein extension $(X,G)$ which defines $P_\gamma^{\olg}$ and hence
$Q_\gamma^g$, then $n$, $k$ and $\gamma$ are restricted by the inequality
\begin{equation}
\Gamma\lp\frac{n}{4} - \frac{k}{2} + \frac{\gamma}{2}\rp \Big/ \Gamma\lp\frac{n}{4} - \frac{k}{2} - \frac{\gamma}{2}\rp > 0,
\label{eq:dimrest}
\end{equation}
where $\Gamma$ is the ordinary Gamma function.  This inequality holds in particular when $k < (n-2\gamma)/2$,
and in this case then there is a unique distributional extension of $u$ on all of $M$ which is still a solution of \eqref{eq:ccQ}.
%or in other words, $u$ extends uniquely to a weak solution of \eqref{eq:ccQ} on all of $M$.
\label{th:fSY}
\end{thm}

\begin{remark}
Recall that $u$ is said to be polyhomogeneous along $\Lambda$ if in terms of any cylindrical coordinate system $(r,\theta,y)$
in a tubular neighborhood of $\Lambda$, where $r$ and $\theta$ are polar coordinates in disks in the normal bundle and
$y$ is a local coordinate along $\Lambda$, $u$ admits an asymptotic expansion
\[
u \sim \sum a_{jk}(y,\theta) r^{\mu_j} (\log r)^k
\]
where $\mu_j$ is a sequence of complex numbers with real part tending to infinity, for each $j$, $a_{jk}$ is nonzero
for only finitely many nonnegative integers $k$, and such that every coefficient $a_{jk} \in \calC^\infty$. The number $\mu_0$
is called the leading exponent if $\Re (\mu_j) > \Re (\mu_0)$ for all $j \neq 0$.  We refer to \cite{Mazzeo:edge-operators}
for a more thorough account of polyhomogeneity.
\end{remark}

\begin{remark}
As we have noted, inequality \eqref{eq:dimrest} is satisfied whenever $k < (n-2\gamma)/2$, and in fact is equivalent to this simpler
inequality when $\gamma = 1$.  When $\gamma=2$, i.e.\ for the standard $Q-$curvature, this result is already known: it is shown in
\cite{Chang-Hang-Yang} that complete metrics with $Q_2 > 0$ and positive scalar curvature must have singular set with
dimension less than $(n-4)/2$, which again agrees with \eqref{eq:dimrest}.
\end{remark}

We also present a few special existence results.  First, the following remark exhibits solutions
coming from Kleinian group theory where $\Lambda$ is nonrectifiable.
\begin{remark}
Suppose that $\gamma\in [1,n/2)$. Let $\Gamma$ be a convex cocompact subgroup of $\mbox{SO}(n+1,1)$ with
Poincar\'e exponent $\delta(\Gamma) \in [1, (n-2\gamma)/2)$. Let $\Lambda\subset S^n$ be the limit set of $\Gamma$.
%If $\dim \Lambda < \frac12 (n-2\gamma)$,
Then $\Omega = S^n \setminus \Lambda$ admits a complete metric $g$ conformal to the round metric and with
$Q_\gamma^g > 0$.
\label{remark:Klein}
\end{remark}
As we explain below, this follows directly from the work of Qing and Raske \cite{Qing-Raske}.

Finally, one can also obtain existence of solutions when $\gamma$ is sufficiently near $1$ and $\Lambda$ is smooth
by perturbation theory.
\begin{thm}
Let $(M^n, [\olg])$ be compact with nonnegative Yamabe constant and $\Lambda$ a $k$-dimensional submanifold
with $k < \frac12 (n-2)$. Then there exists an $\epsilon > 0$ such that if $\gamma \in (1-\epsilon, 1 + \epsilon)$,
there exists a solution to the fractional singular Yamabe problem \eqref{eq:ccQ} with $Q_\gamma > 0$ which is complete
on $M \setminus \Lambda$.
\label{th:pert}
\end{thm}

% Along the way, we also provide a better understanding of distributional solutions for the Yamabe problem. Indeed, weak solutions must satisfy a growth bound near the singular set. This is done in section \label{section-growth}.\\
Our final result is a growth estimate for weak solutions that are singular on $S^n\backslash\Omega$. Our result is not very strong in the sense that we do need to require that $u$ is a weak solution in the whole $S^n$. However, it provides the first insight into a general theory of weak solutions on subdomains of $S^n$.

\begin{prop}\label{prop:growth}
Let $g_c$ be the standard round metric on $S^n$, and $(B^{n+1},G)$ the Poincar\'e ball model of hyperbolic space, which
has $(S^n,[g_c])$ as its conformal infinity. Let $g = u^{\frac{4}{n-2\gamma}}g_c$ be a complete metric on a dense subdomain of
the sphere, $\Omega = S^n \setminus \Lambda$, with $Q_\gamma^g$ equal to a positive constant, and such that $u$ is a
distributional solution to
\be
\label{problem3}P_\gamma^{g_c} u=u^{\frac{n+2\gamma}{n-2\gamma}}
\ee
on $S^n$ (with $u$ finite only on $\Omega$). Then, for all $z\in\Omega$,
\[
u(z)\leq \frac{C}{d_{g_c}(z,\Lambda)^{\frac{n-2\gamma}{2}}},
\]
where $C$ depends only on $n$ and $\gamma$.
\end{prop}

There are many interesting questions not addressed here. For example, we point out again that there is not yet
a good definition of the family $P_\gamma^g$ on an arbitrary complete manifold $(\Omega,g)$. Provided one is able
to make this definition, it would then be useful to compute the $L^2$-spectrum of $P_\gamma^g$, even for some
specific examples such as $\HHs^n$ or $\HHs^{k+1} \times S^{n-k-1}$. Finally, it would also be important to obtain
the correct generalization of the Schoen-Yau theorem for the operators $P_\gamma$. We hope to address
these and other problems elsewhere.

\section{Fractional conformal Laplacians}

We now provide a more careful description of the construction of the family of conformally covariant operators $P_\gamma$,
and also give two alternate definitions of these operators in the flat case to provide some perspective.

As we have described in the introduction, Graham and Zworski \cite{GZ} discovered a beautiful connection
between the scattering theory of the Laplacian on an asymptotically hyperbolic Einstein manifold and the
GJMS operators on its conformal infinity. Let $(M,g)$ be a compact $n$-dimensional Riemannian manifold.
Suppose that $X$ is a smooth compact manifold with boundary, with $\del X = M$, and denote by $x$ a defining
function for the boundary, i.e.\ $x \geq 0$ on $X$, $x = 0$ precisely on $\del X$ and $dx \neq 0$ there. A metric
$G$ on the interior of $X$ is called conformally compact if $x^2 G = \overline{G}$ extends as a smooth nondegenerate
metric on the closed manifold with boundary. It is not hard to check that $G$ is complete and, provided that
$|dx|_{\overline{G}} = 1$ at $\del X$, the sectional curvatures of $G$ all tend to $-1$ at `infinity'.  The metric
$G$ is called Poincar\'e-Einstein if it is conformally compact and also satisfies the Einstein equation $\Ric^G = - n G$.
As we have explained, it is only necessary to consider asymptotically Poincar\'e-Einstein metrics; by definition,
these are conformally compact metrics which satisfy $\Ric^G = -nG + \calO(x^N)$ for some suitably large $N$
(typically, $N > n$ is sufficient).

The conformal infinity of $G$ is the conformal class of $\left. \overline{G} \right|_{T\del X}$; only the conformal class is
well defined since the defining function $x$ is defined up to a positive smooth multiple. If $g$ is any representative
of this conformal class, then there is a unique defining function $x$ for $M$ such that $G = x^{-2}(dx^2 + g(x))$ where
$g(x)$ is a family of metrics on $M$ (or rather, the level sets of $x$), with $g(0)$ the given initial metric.

We now define the scattering operator $S(s)$ for $(X,G)$. Fix any $f_0 \in \calC^\infty(M)$; then for all but a discrete set of
values $s \in \CC$, there exists a unique generalized eigenfunction $u$ of the Laplace operator on $X$ with eigenvalue $s(n-s)$.
In other words, $u$ satisfies
\be\label{eigenvalue-problem}
\left\{\begin{split}&(\Delta_G - s(n-s))u = 0\\
&u = f x^{n-s} + \tilde{f} x^{s}, \quad \mbox{for some}\ f, \tilde{f} \in \calC^\infty(\overline{X}) \quad
\mbox{with}\ \left. f \right|_{x=0} = f_0.
\end{split}\right.\ee
By definition, $S(s)f_0 = \tilde{f}|_{x=0}$. This is an elliptic pseudodifferential operator of order $2s - n$
which depends meromorphically on $s$; it is known to always have simple poles at the values $s = n/2, n/2 +1 ,
n/2 + 2, \ldots$. These locations are independent of $(X,G)$, hence are called the trivial poles of the
scattering operator. $S(s)$ has infinitely many other poles which are of great interest in other investigations,
but do not concern us here.  Letting $s=n/2+\gamma$, we now define
\begin{equation}
P_\gamma^g = 2^{2\gamma} \frac{\Gamma(\gamma)}{\Gamma(-\gamma)} S\lp\frac{n}{2} + \gamma\rp;
\end{equation}
because of these prefactors, one has that the principal symbol is
\begin{equation}
\sigma_{2\gamma}(P_\gamma^g) = |\eta|^{2\gamma}_g.
\end{equation}
The scattering operator satisfies a functional equation, $S(s) S(n-s) = \mbox{Id}$, which implies that
\begin{equation}
P_\gamma \circ P_{-\gamma} = \mbox{Id}.
\end{equation}
Finally, it is proved in \cite{GZ} that the operators $P_\gamma^g$ satisfy the conformal covariance equation \eqref{eq:ccfcL}.

This definition of the operators $P_\gamma$ depends crucially on the choice of the Poincar\'e-Einstein filling $(X,G)$.
Graham and Zworski point out that it is only necessary that the metric $G$ satisfy the Einstein equation to sufficiently
high order as $x \to 0$ in order that the properties of the $P_\gamma$ listed above be true (for $\gamma$ in a finite
range which depends on the order to which $G$ satisfies the Einstein equation). As we have discussed in the
introduction, it is always possible to find such metrics, and we suppose that one has been fixed.
% This is important since the Poincar\'e-Einstein filling problem is difficult and there are no good general existence theorems,
% but it is always possible to find an asymptotically Poincar\'e-Einstein filling (so long as there exists a smooth
% compact manifold $X$ with $\del X = M$).  In any case, our aims here are quite different, so we assume henceforth
% that each manifold $M$ that we consider is a boundary, and that we have fixed an asymptotically Poincar\'e-Einstein
% metric $G$ on the interior.

Let us now address the issue of how to define $P_\gamma^g$ and $Q_\gamma^g$ when $\Omega$ is a dense open set
in a compact manifold $M$ and $g$ is complete and conformal to a metric $\olg$ which extends to all of $M$. (As usual,
we assume that $(M,\olg)$ has an asymptotically Poincar\'e-Einstein filling).  There is no difficulty
in using the relationship \eqref{eq:ccfcL} to define $P_\gamma^gf$ when $f \in \calC^\infty_0(\Omega)$ .
From here one can use an abstract functional analytic argument to extend $P_\gamma^g$ to act on any $f \in L^2(\Omega, dV_g)$.
Indeed, it is straightforward to check that the operator $P_\gamma^g$ defined in this way is essentially self-adjoint on
$L^2(\Omega, dV_g)$ when $\gamma$ is real. However, observe that $P_\gamma = \Delta_g^{\gamma} + K$, where
$K$ is a pseudo-differential operator of order $2\gamma-1$. Furthermore, $\Delta^{\gamma}$ is self-adjoint by
the functional calculus, so we can appeal to a classical theorem, see \cite{RS}, which states that a lower order
symmetric perturbation of a self-adjoint operator is essentially self-adjoint.

A separate, but also very interesting issue, is whether $Q_\gamma^g$ is a positive constant implies that the conformal factor
$u$ is a weak solution of \eqref{eq:ccQ} on all of $M$. This is true (with some additional hypotheses) when $\gamma = 1$,
cf. \cite{SY}.

We conclude this section with two alternate definitions of the operators $P_\gamma^{\olg}$ in the special
case where $(M,[\olg]) = \RR^n$ with its standard flat conformal class.

The canonical Poincar\'e-Einstein filling in this case is the hyperbolic space $X = \RR^{n+1}_+ = \RR^+_x \times \RR^n_y$ with
metric $G = x^{-2}(dx^2 + |dy|^2)$.

Since $\olg$ is flat, we have $P_\gamma^{\olg} = \Delta_{\olg}^\gamma$, and this can be written in either of the two equivalent forms:
\begin{eqnarray*}
\Delta^\gamma f(y) & = &  (2\pi)^{-n} \int_{\RR^n} e^{iy \eta} |\eta|^{2\gamma} \hat{f}(\eta)\, d\eta  \quad \mbox{or} \\
& = & \mbox{P.V.}\int_{\RR^n}\frac{f(y)-f(\tilde{y})}{|y-\tilde{y}|^{n+2\gamma}}\, d\tilde{y}.
\end{eqnarray*}
Both formul\ae\ can be regularized so as to hold for any given $\gamma$.

One other way that $\Delta^\gamma$ arises is as a generalized Dirichlet to Neumann map; this definition is essentially the
same as the one above involving the scattering operator (indeed, the point of view in geometric scattering theory
is that the scattering operator \emph{is} simply the Dirichlet to Neumann operator at infinity), but as recently
rediscovered by Chang-Gonzalez \cite{CG} in relation to the work on (Euclidean) fractional Laplacians by Caffarelli
and Silvestre \cite{cafS}, it is sometimes helpful to consider the equation in a slightly different form.
In the following result, let $(X,G)$ be an asymptotically Poincar\'e-Einstein filling of the compact manifold $(M^n,[\bar g])$.
Fix a representative $\bar{g}$ of the conformal class on the boundary and let $x$ be the boundary defining function on $X$
such that $g = x^{-2}(dx^2 + \bar{g}_x)$ with $\bar{g}_0 = \bar{g}$. Also, write $\bar{G} = x^2 G$; this is an incomplete
metric on $\overline{X}$ which is smooth (or at least polyhomogeneous) up to the boundary.
\begin{prop}(\cite{CG})\label{prop:extension}
Let $U = x^{\frac{n}{2}-\gamma} u$ and
\[
E :=\Delta_{\bar G}\lp x^{\frac{1-2\gamma}{2}}\rp x^{\frac{1-2\gamma}{2}}+\lp\gamma^2-\tfrac{1}{4}\rp x^{-1-2\gamma}+
\tfrac{n-1}{4n} R_{\bar G} x^{1-2\gamma}.
\]
Then, for any $f_0\in\mathcal C^{\infty}(M)$, the eigenvalue problem \eqref{eigenvalue-problem} is equivalent to
\be\label{divergence-equation}
\left\{\begin{split}
- \divergence \lp x^{1-2\gamma} \nabla U\rp + E U &=0\quad \mbox{ on }(X, \bar{G}), \\
\left. U \right|_{x=0}&=f_0\quad \mbox{on }M,
\end{split}\right.
\ee
where the divergence and gradient are taken with respect to $\bar{G}$. Moreover,
\[
P_\gamma^{\bar g} (f_0)=d_\gamma \lim_{x\to 0} x^{1-2\gamma}\partial_x U
\]
for some nonzero constant $d_\gamma$ depending only on $\gamma$ and $n$.
\end{prop}
The Euclidean version of this result (where $(X,G)$ is the hyperbolic upper half-space) was the one studied by
Caffarelli and Silvestre. The main advantage in this reformulation is that certain estimates are more
transparent from this point of view.

% ============================================================================
% ============================================================================

\section{Proofs}
We now turn to the proofs of Theorems \ref{th:fSY}  and \ref{th:pert}, and Remark  \ref{remark:Klein}.

\subsection{Dimensional restrictions on singular sets}

The idea for the proof of Theorem \ref{th:SY} is straightforward: let $u$ be a polyhomogeneous distribution on $M$ with singular
set along the smooth submanifold $\Lambda$. Suppose that the leading term in the expansion of $u$ is $a(y) r^{-n/2 + \gamma}$.
Then by a standard result in microlocal analysis \cite{DH}, the function $P_\gamma^{\olg}u$ is again polyhomogeneous and
has leading term $b(y) r^{-(2\gamma+n)/2}$, where $b(y) = \lambda a(y)$ for some constant $\lambda$. Now, if $u$ is
a conformal factor for which $g = u^{4/(n-2\gamma)}\olg$ has $Q_\gamma^g > 0$, then $P_\gamma^{\olg}u > 0$, which implies
that $\lambda > 0$.  So we must compute $\lambda$ to obtain \eqref{eq:dimrest}.

This microlocal argument states that if $u$ is polyhomogeneous, then the leading term of $P_\gamma^{\olg}u$ can be computed
using the symbol calculus (for pseudodifferential operators and for polyhomogeneous distributions), and more specifically,
that the principal symbol of $P_\gamma^{\olg}u$ is equal to the product of the principal symbols of $P_\gamma^{\olg}$ and
that of $u$. (Note that the principal symbol of a distribution conormal to a submanifold $\Lambda$ is computed in terms of
the Fourier transform in the fibres of $N\Lambda$.) In the present setting, this implies that the constant $\lambda$
is the same as for the model case when $M = S^n$ and $\Lambda$ is an equatorial $S^k$, so we now focus on this special case.

Transform $S^n$ to $\RR^n$ by stereographic projection, so that $\Lambda$ is mapped to a linear subspace $\RR^k$
and $\olg$ is the flat Euclidean metric (which we henceforth omit from the notation). Write $\RR^n \ni y  = (y',y'') \in
\RR^k \times \RR^{n-k}$, so that (in this model case) $u(y) = |y''|^{-n/2 + \gamma}$ for the singular metric $u^{\frac{4}{n-2\gamma}}\bar g$; then
\begin{multline*}
P_\gamma u(y) = \Delta^\gamma u (y) = (2\pi)^{-n} \int_{\mathbb R^n\times \mathbb R^n} e^{i(y-\tilde{y})\cdot \eta} |\eta|^{2\gamma} |y''|^{-n/2 + \gamma}\, d\tilde y d\eta
\\ = (2\pi)^{k-n} \int_{\RR^{n-k}\times\RR^{n-k} } e^{i(y'' - \tilde{y}'')\cdot\eta''} |\tilde{y}''|^{-n/2  + \gamma}\, d\tilde y'' d\eta''.
\end{multline*}

Now recall a well-known formula for the Fourier transform of homogeneous distributions in $\RR^N$:
\[
\int_{\mathbb R^N} e^{-iz \cdot \zeta}  |z|^{-N + \alpha} \,dz= c(N,\alpha) |\zeta|^{-\alpha},
\]
where
\[
c(N,\alpha) = \pi^{\alpha - N/2}\frac{\Gamma(\alpha/2)}{\Gamma((N -\alpha)/2)}.
\]
Applying this formula with $N = n-k$ (and replacing $y''$ by $y$ and $\eta''$ by $\eta$, for simplicity) yields first that
\[
\int_{\RR^N}  e^{-iy \cdot \eta} |y|^{-n/2 + \gamma}\, dy = c\lp n-k, \frac{n}{2} - k + \gamma\rp\, |\eta|^{-\frac{n}{2} + k - \gamma},
\]
then, multiplying by $|\eta|^{2\gamma}$ and taking inverse Fourier transform we obtain
\[
\frac{1}{(2\pi)^{n-k}} c\lp n-k, \frac{n}{2} - k + \gamma\rp c\lp n-k, \frac{n}{2} + \gamma\rp |y|^{-\frac{n}{2} - \gamma}.
\]
Altogether then,  the multiplicative factor $\lambda$ is equal to
\[
2^{k-n} \pi^{k-n + 2\gamma} \frac{\Gamma\lp\frac12( \frac{n}{2} - k + \gamma)\rp}{\Gamma(\frac12 \lp\frac{n}{2} - \gamma)\rp}
\frac{ \Gamma\lp \frac12 (\frac{n}{2} + \gamma)\rp}{\Gamma\lp \frac12 (\frac{n}{2} - k - \gamma)\rp}.
\]
Discarding the factors which are always positive (which includes $\Gamma(n/4 - \gamma/2)$ since $\gamma < n/2$),
we obtain \eqref{eq:dimrest}.

It is unfortunately slightly messy to write down the entire set of values of $k$ and $\gamma$ for which
\eqref{eq:dimrest} holds. However, if $k < (n-2\gamma)/2$, then both $\frac{n}{2} - k \pm \frac12 \gamma > 0$.
Furthermore, if $\gamma = 1$ and $A := \frac{n}{2} - k < \frac12$, then the $\Gamma$ function always takes on
values with different signs at $A + \gamma/2$ and $A - \gamma/2$.  More  generally, if we fix $n$ and $k$
and let $\gamma$ increase from $0$ to $n/2$, then $\Gamma(A + \gamma/2)/\Gamma(A - \gamma/2) = 1$
when $\gamma = 0$; $\Gamma(A - \gamma/2)$ changes sign every time $\gamma$ increases by $2$, whereas
$\Gamma(A + \gamma/2)$ also changes similarly, but only for $\gamma$ in the range $(0, -2A)$, where $A < 0$.

To prove the final statement of the theorem, note that if $-\gamma - n/2$, the leading exponent of $P_\gamma^{\olg} u$, is
greater than $k-n$, the codimension of $\Lambda$, then $P^{\bar g}_\gamma u$ cannot have any mass supported on $\Lambda$,
which means that $u$ is a weak solution of $P_\gamma^{\olg} u = Q_\gamma^g u^{(n+2\gamma)/(n-2\gamma)}$ on all of $M$.

\subsection{Kleinian groups}
We now turn to a special case where this problem has a direct relationship to hyperbolic geometry. Let $\Gamma$ be a convex
cocompact group of motions acting on $\HHs^{n+1}$. Thus $\Gamma$ acts discretely and properly discontinuously
on hyperbolic space, is geometrically finite and contains no parabolic elements. Its domain of discontinuity
is the maximal open set $\Omega \subset S^n$ on which the action of $\Gamma$ extends to a discrete and
properly discontinuous action; by definition of convex cocompactness, the quotient $\Omega/\Gamma = Y$
is a compact manifold with a locally conformally flat structure. The complement $S^n \setminus \Omega  = \Lambda$
is the limit set of $\Gamma$. Furthermore, the manifold $X = \HHs^{n+1}/\Gamma$ with its hyperbolic metric
is Poincar\'e-Einstein with conformal infinity $Y$ with the conformal structure induced from $S^n$. We use these
canonical fillings to define $P_\gamma$ and $Q_\gamma$.

In \cite{Qing-Raske}, Qing and Raske attack the problem of finding metrics of constant $Q_k$ curvature (with $k < n/2$
an integer). Their method involves finding metrics of constant $Q_\gamma$ curvature for all $1 \leq \gamma \leq k$.
They rephrase the problem $P_\gamma u = Q_\gamma u^{(n+2\gamma)/(n-2\gamma)}$ in the equivalent form $u =
P_{-\gamma} (Q_\gamma u^{(n+2\gamma)/(n-2\gamma)})$. The advantage of this modification is that $P_{-\gamma}$ is
a pseudodifferential operator of negative order, and its Schwartz kernel can be obtained by summing the
translates of the Schwartz kernel of $P_{-\gamma}$ on $S^n$ over the group $\Gamma$. This sum converges provided
the Poincar\'e exponent of $\Gamma$ is less than $\frac{n-2\gamma}{2}$, and because this is a convergent sum
one directly obtains explicit control and positivity of this operator. Fixing $1 \leq \gamma < n/2$ and
restricting to convex cocompact groups $\Gamma$ with Poincar\'e exponent in this range, they are able to
prove that if $Y = \Omega/\Gamma$ has positive Yamabe type, then it admits a metric of constant positive
$Q_\gamma$ curvature.

The proof of Remark \ref{remark:Klein} follows directly from this by lifting the conformal factor and solution metric to
$\Omega \subset S^n$. Namely, by the theorem of Schoen and Yau, the developing map of $Y$ is injective from the
universal cover $\tilde{Y}$ to $\Omega$, and the solution metric $g$ on $Y$ lifts to a complete metric $\tilde{g}$
on $\Omega$ of the form $u^{(n+2\gamma)/(n-2\gamma)} \olg$, where $\olg$ is the standard round metric. Using the
bound on the Poincar\'e exponent and the compactness of $Y$, standard lattice point counting arguments
show that $u(p) \leq c\, \mbox{dist}_{\olg}\, (p,\Lambda)^{(2\gamma-n)/2}$. This shows that not only is $u$
a solution of the modified integral equation, but is also a weak solution of \eqref{eq:ccQ} on all of $S^n$
and that $Q_\gamma^g$ is constant. Finally, by Patterson-Sullivan theory, the dimension of the limit set $\Lambda$
is precisely the Poincar\'e exponent $\delta(\Gamma)$. In other words, we have produced a solution to the fractional
singular Yamabe problem with exponent $\gamma \in [1,n/2)$ and with singular set of dimension less than
$n/2 - \gamma$.

The Qing-Raske theorem is not stated for the remaining cases $\gamma \in (0,1)$; it is plausible that their proof
may be adapted to work then, and hence the lifted solution would also give a solution to our problem
also for $\gamma$ in this range, but we do not claim this. Note, however, the results in \S 4 below concerning
growth estimates for solutions of this equation for this range of $\gamma$.

\subsection{Perturbation methods}

We come at last to the perturbation result.  We deduce existence of solutions for the fractional singular
Yamabe problem for values of $\gamma$ near $1$ from the the general existence result
in \cite{MP} for the singular Yamabe problem with $\gamma = 1$. Let $(M,\olg)$
and $\Lambda$ be a submanifold of dimension $k$ as in the statement of the theorem. (Slightly more
generally, we could let the different components of $\Lambda$ have different dimensions, but for
simplicity we assume that $\Lambda$ is connected.)  Then there is a function
$u$ on $M \setminus \Lambda$ such that $g = u^{4/(n-2)}\olg$ is complete and its scalar curvature
$Q_1^g$ is a positive constant. Moreover, it is known that the linearization of the equation \eqref{eq:cL}
at any one of the solutions $u$ constructed in \cite{MP} is surjective on appropriate weighted H\"older spaces.

In the following we phrase this rigorously and parlay this surjectivity into an existence theorem for $\gamma$
near $1$ using the implicit function theorem. Let $\calT^\sigma\Lambda$ denote the tube of radius $\sigma$
(with respect to $\olg$) around $\Lambda$; this is canonically diffeomorphic to the neighourhood of
radius $\sigma$ in the normal bundle $N\Lambda$ for $\sigma$ sufficiently small, and we use this to transfer
cylindrical coordinates $(r,y,\theta) \in [0,\sigma) \times \calU_y \times S^{n-k-1}$ in a local trivialization of
$N\Lambda$ to Fermi coordinates in $\calT^\sigma\Lambda$.

We use these coordinates to define weighted H\"older spaces with a certain dilation covariance property.
For $w \in \calC^0(\calT^\sigma\Lambda)$, let
\[
\|w\|_{e, 0,\alpha,0}=\sup_{z \in \mathcal T^\sigma \Lambda} |w| +\sup_{z, \tilde z \in
\mathcal T^\sigma \Lambda} \frac{(r+\tilde r)^\alpha |w(z)-w(\tilde z)|}{|r-\tilde r|^\alpha +
|y-\tilde y|^\alpha+(r+\tilde r)^\alpha |\theta -\tilde \theta|^\alpha}.
\]
and denote by $\calC^{0,\alpha}_e(M \backslash \Lambda)$ the space of all functions $w \in \calC^0(T^\sigma\Lambda)$
such that this norm is finite. The initial subscript $e$ in the norm signifies that these are `edge' H\"older spaces. Next,
$\calC^{k,\alpha}_e(M\backslash \Lambda)$ denotes the subspace of $\calC^k(M \backslash \Lambda)$ on which the norm
\[
\|w\|_{k,\alpha,0}=\|w\|_{k,\alpha,M_{\sigma/2}} + \sum_{j=0}^k \|\nabla^j w \|^{\mathcal T^\sigma \Lambda}_{e, 0, \alpha}
\]
is finite, where $M_{\sigma/2}= M \backslash \mathcal T_{\sigma/2}^\Lambda$. Finally, for $\nu \in \mathbb R$, let
\[
\calC_\nu^{k,\alpha}(M \backslash \Lambda)= \left \{ w=r^\nu \overline w : \,\, \overline w \in \calC^{k,\alpha}_e
(M \backslash \Lambda) \right \},
\]
with corresponding norm $|| \cdot ||_{e,k,\alpha,\nu}$.

Fixing $Q_\gamma^g = 1$, the linearization of $u \mapsto P_\gamma^{\olg}u - u^{ (n+2\gamma)/ (n-2\gamma)}$
is the operator
\[
v \mapsto L_\gamma v := P_\gamma^{\olg} v - \frac{n+2\gamma}{n-2\gamma} u^{\frac{4\gamma}{n-2\gamma}} v.
\]
Let $u$ be one of the solutions to the singular Yamabe problem ($\gamma = 1$) on $M \setminus \Lambda$
constructed in \cite{MP}. It is proved there that the solution $u$ has the form $u = c_1 r^{1 - n/2}(1 + v)$, where
$v \in \calC^{2,\alpha}_\nu$ for any $0 < \nu < k/2$ and $c_1 > 0$ depends only on the dimensions $k$ and $n$;
furthermore, the mapping
\[
L_1: \calC^{2,\alpha}_{\nu}(M \setminus \Lambda) \longrightarrow \calC^{0,\alpha}_{\nu -2}(M \setminus \Lambda)
\]
is surjective for $\nu$ in this same range.

We claim that for $\gamma$ sufficiently close to $1$, and for $\nu \in (\eta, k/2 - \eta)$, where $\eta > 0$
is some small fixed number, the mapping
\[
L_\gamma: \calC^{2,\alpha}_{\nu}(M \setminus \Lambda) \longrightarrow \calC^{0,\alpha + 2(1-\gamma)}_{\nu -2\gamma}
(M \setminus \Lambda)
\]
is also bounded and surjective. The boundedness follows by an interpolation argument. Indeed, the spaces $\calC^{k,\alpha}_0$
have interpolation properties which are identical to those for the ordinary H\"older spaces since they \emph{are} just
the standard H\"older spaces for the complete metric $\tilde{g} = \olg/r^2$; a minor adjustment shows that the
addition of the weight factor behaves as expected.  The assertion about the boundedness of $L_\gamma$
is clearly true for $\gamma = 0, 1, 2$, and hence by interpolation is true for all $\gamma$ close to $1$.
(It is true for the full range of $\gamma \in (0,2)$ if one makes the standard change, replacing the H\"older
space by a Zygmund space, when $\alpha + 2(1-\gamma)$ is an integer.)  This also follows from
\cite{Mazzeo:edge-operators} because $r^{2\gamma}L_\gamma$ is a pseudodifferential edge operator of
order $2\gamma$. Similarly, surjectivity follows from the construction of a parametrix for $L_\gamma$
in the edge calculus, from \cite{Mazzeo:edge-operators} again. This proves that $L_\gamma$ is Fredholm,
and since it is surjective at $\gamma = 1$, it must remain surjective for values of $\gamma$ which are
close to $1$. We write its right inverse as $G_\gamma$.

Now consider the mapping
\[
(\gamma, c, v) \longmapsto N(\gamma,c,v) := G_\gamma
(P_\gamma^{\olg} cr^{\gamma - n/2} (1 + v) - (c r^{\gamma - n/2}(1+v))^{\frac{n+2\gamma}{n - 2\gamma}}).
\]
If $u_1 = c_1 r^{1-n/2}(1+v_1)$ is the solution to the singular Yamabe problem from \cite{MP}, then $N(1,c_1,v_1) = 0$.
Let $c\in (c_1 -\e, c_1 + \e)$, and similarly, $v - v_1$ lie in a ball of radius $\e$ about $0$ in
$\calC^{2,\alpha}_\nu$. Clearly $\left. D_v N\right|_{(1,c_1,v_1)} = G_1 L_1 = \mbox{Id}$. The implicit function theorem now
applies to show that for every $(\gamma,c)$ near to $(1,c_1)$, there exists a unique $v_\gamma \in \calC^{2,\alpha}_\nu$
with norm less than $\e$ such that
$u_\gamma = c r^{\gamma - n/2} (1 + v_\gamma)$ is a solution of the fractional singular Yamabe problem with singular set $\Lambda$.

% ============================================================================================
% ============================================================================================

\section{Growth estimates for weak solutions on $S^n$}
In this final section we furnish the proof of Proposition \ref{prop:growth}: if $\gamma \in (0,1)$ and $\Omega \subset S^n$ is dense, then
any weak solution of the fractional singular Yamabe problem
\be\label{Yamabe1}
P_\gamma^{g_c}(u)=u^{\frac{n+2\gamma}{n-2\gamma}} \ \ \mbox{ in }S^n,\quad u>0, \quad u \mbox{ singular along }S^n\backslash\Omega
\ee
satisfies a general growth estimate. This is a direct adaptation of Schoen's proof (which is written out in full in
\cite{Pollack:compactness}) for the case $\gamma = 1$.

We first comment on the local regularity for solutions of \eqref{problem3}. There are several ways to deduce the
necessary estimates. The path we follow uses the equivalence, as described in \ref{prop:extension}, of
\eqref{problem3} with the extension problem \eqref{divergence-equation}:
\be\label{divergence-equation2}
\left\{\begin{split}
-\divergence \lp x^{1-2\gamma} \nabla U\rp + E(x) U &=0, \quad \mbox{in } (X,\bar G),\\
-y^{1-2\gamma}\partial_x U&=c_{n,\gamma} U^{\frac{n+2\gamma}{n-2\gamma}}, \quad\mbox{on } x=0;
\end{split}\right.
\ee
here $U=x^{\frac{n}{2}-\gamma} u$, $\bar G=x^2 G$.

From this point of view, we can use the linear regularity theorem \cite[Theorem 7.14]{Mazzeo:edge-operators} to
prove that $U$ is smooth up to $x=0$ away from $\Lambda$. This can also be deduced using standard elliptic
estimates for the pseudodifferential operator $P_\gamma^{g_c}$, but we refer also to more classical sources from
which this can also be deduced, in particular the paper by \cite{Fabes-Kenig-Serapioni:local-regularity-degenerate};
we also refer to more recent references \cite{Cabre-Sire:I} (where many properties of the solution are written down), and
\cite{fractional-Yamabe} (which holds for more general ambient metrics). In particular, from these last papers, one
has that Schauder and local  $L^p\to L^\infty$ estimates hold, and the equation also satisfies the standard maximum principles.

Fix $z_0\not\in\Lambda$ and choose $\sigma < \dist_{g_c}(z_0,\Lambda)$. For simplicity, write $\rho(z) :=\dist_{g_c}(z, z_0)$. Now define
\[
f(z):=\lp\sigma-\rho(z)\rp^{\frac{n-2\gamma}{2}}u(z);
\]
note that $f = 0$ on $\partial B_\sigma(x_0)$.

It suffices to show that $f(z)\leq c$ for some $c > 0$ and for all $z\in B_\sigma(z_0)$ since if we choose $\sigma=\dist(z_0,\Lambda)/2$,
then $f(z_0)=\sigma^{\frac{n-2\gamma}{2}} u(z_0)$, and hence
\[
u(z_0)\leq \frac{c}{d(z_0,\Lambda)^{\frac{n-2\gamma}{2}}},
\]
which would finish the proof.

We prove this claim by contradiction. Assume that no such $c$ exists. Then there exists a sequence
$\{u_m,\Lambda_m, \sigma_m, z_{0,m}, z_m\}$ such that for all $m$, $f_m$ attains its maximum
in $B_{\sigma_m}(z_{0,m})$ at $z_m$, and
\[
f(z_m):=(\sigma_m-\dist(z_m, z_{0,m}))^{\frac{n-2\gamma}{2}} u_m(z_m) > m.
\]
Since $(\sigma_m-\dist(z_m, z_{0,m}))^{\frac{n-2\gamma}{2}}\leq \sigma_m^{\frac{n-2\gamma}{2}} \leq C$ for all $m$,
we see that necessarily $u_m(z_m)\to\infty$.

Let $z$ be a system of Riemann normal coordinates centered at $z_m$, so that the corresponding metric coefficients
satisfy $(g_c)_{ij}=\delta_{ij}+O(\abs{z}^2)$. (As $m$ varies, these coordinate systems also vary, but there is no
reason to include this in the notation.)  Set $\lambda_m=\lp u_m(z_m)\rp^{\frac{2}{n-2\gamma}}$; we consider the dilated
coordinate system $\zeta =\lambda_m z$, the corresponding sequence of metrics $\hat g_m$, where
$$
g_c=\lambda_m^{-2} \sum_{i,j=1}^n (g_c)_{ij}(\zeta/\lambda_m) d\zeta^i d\zeta^j:= \lambda_m^{-2} \hat g_m,
$$
and finally the dilated family of solutions
$$
v_m(\zeta):=\lambda_m^{-\frac{n-2\gamma}{2}} u_m\lp \frac{\zeta}{\lambda_m}\rp.
$$
By construction, $v_m(0) = 1$ for all $m$, and
$$
g_m:=u_{m}^{\frac{4}{n-2\gamma}}g_c=v_{m}^{\frac{4}{n-2\gamma}}\hat g_m.
$$
We show below that $\hat{g}_m$ and $v_m$ are defined on an expanding sequence of balls on $\RR^n$, and it is then
clear that $\hat{g}_m$ converges to the Euclidean metric uniformly in $\calC^\infty$ on any compact subset.

Let $r_m = \frac12 (\sigma_m - \rho(z_m))$, or equivalently, $\rho_m(z_m) + 2r_m = \sigma_m$. Then
\[
\sigma_m - \rho(z_m) \geq \sigma_m - \rho_m(z_m) - r_m = r_m
\]
on the ball $\mbox{dist}_{g_c}(z,z_m) < r_m$, and hence on this same ball,
\[
u_m(z) \leq \left(\frac{\sigma_m - \rho_m(z_m)}{\sigma_m - \rho_m(z)}\right)^{\frac{n+2\gamma}{n-2\gamma}} u_m(z_m) \leq c \, u_m(z_m),
\quad c=2^{\frac{n-2\gamma}{2}}.
\]
The corresponding ball in rescaled coordinates contains
$\{\zeta: |\zeta| < m^{\frac{2}{n-2\gamma}}\}$, hence has radius tending to infinity. By construction, $v_m(z)\leq c$ on
this entire ball, and $v_m(0) \equiv 1$.  Since these functions are uniformly bounded and satisfy the converging
set of elliptic pseudodifferential equations
\be
\label{problem1}P_{\gamma}^{\hat g_m} v_m=v_m^{\frac{n+2\gamma}{n-2\gamma}},
\ee
we conclude using the local regularity theory (which is straightforward since $v_m$ is bounded) that $v_m$ is
bounded in $\calC^{2,\alpha}$ of every compact set, and hence we can extract a convergent subsequence.
We thus obtain a smooth solution $v$ to the `flat' equation
\be
\label{problem2}-(\Delta_{\mathbb R^n})^\gamma v=v^{\frac{n+2\gamma}{n-2\gamma}}\quad \mbox{in }\mathbb R^n.
\ee
Since each $v_m > 0$, we see that $v \geq 0$, but $v \not\equiv 0$ since $v(0) = 1$. There is a maximum
principle for this equation \cite[Corollary 3.6]{fractional-Yamabe} when $0 < \gamma \leq 1$, so we
conclude that $v > 0$ on all of $\RR^n$.

There is a complete characterization of positive solutions of  \eqref{problem2}, \cite[\S 5]{fractional-Yamabe}).
They are the extremal functions for the
embedding $H^\gamma(\mathbb R^n)\hookrightarrow L^{\frac{2n}{n-2\gamma}}(\mathbb R^n)$, and are necessarily of the form
$$
v(z)=C\lp\frac{\mu}{\abs{z-z_0}^2+\mu^2}\rp^{\frac{n-2\gamma}{2}},
$$
for some $\mu,c>0$ and $z_0\in\mathbb R^n$ (these are well known ``bubbles'').

The argument is completed using Theorem  \ref{thm:convex-boundary} below, which states that a small ball in $\Omega_m=S^n\setminus\Lambda_m$ must have a concave boundary with respect to to $g_m$ for $m$
sufficiently large. This is a contradiction to the already known limiting form of the $v_m$. The proof of Proposition \ref{prop:growth} is thus completed.

\qed\\

Note that the previous arguments do not require that $u$ is a weak solution in the whole $S^n$. The only place where this strong hypothesis is required is in the following convexity claim:

\begin{thm}\label{thm:convex-boundary}
In the same hypothesis as in Proposition \ref{prop:growth}, any open ball $B$ (with respect to $g_c$) with $\bar B\subset\Omega$, has boundary $\partial B$ which is geodesically convex with respect to $g$.
\end{thm}

This result was proved in \cite{Schoen:number-metrics} for constant scalar curvature metrics, and also in the case
$\gamma\in(1,n/2)$ for locally conformally flat manifolds satisfying some extra conditions by \cite{Qing-Raske}.
The crucial step is the application of the Alexandroff moving plane method. As we show here, the same ideas work
in the fractional case. The moving plane method has been successfully applied to fractional order operators in
\cite{Qing-Raske} and \cite{Chen-Li-Ou:classification-solutions}, at least when the equation is rewritten as an
integral equation.  However, the proof in the present seting is simpler because of the equivalent formulation \eqref{divergence-equation2} and the precise asymptotics
\eqref{behavior-infinity}, so we include the details for the reader's convenience. Our proof follows the classical
arguments for the Laplacian by Gidas-Ni-Nirenberg in \cite{Gidas-Ni-Nirenberg},
\cite{Gidas-Ni-Nirenberg:symmetry-max-principle}. \\

For simplicity, we denote $P_\gamma:=P^{|dx|^2}_\gamma$. Let $v$ be a distributional solution of
\be\label{equation-euclidean}P_\gamma v=v^{\frac{n+2\gamma}{n-2\gamma}}\quad\mbox{in }\mathbb R^n.\ee
We will apply the Alexandroff reflection with respect to the planes $S_\lambda:=\{x\in\mathbb R^n: x^n=\lambda\}$.
Let $\Sigma_\lambda:=\{x\in\mathbb R^n \,: \,x^n>\lambda\}$ be the hyperplane lying above $S_\lambda$. Given $x=(x^1,\ldots,x^n)\in\Sigma_\lambda$, define $x^\lambda$ to be the reflection of $x$ with respect to the hyperplane $S_\lambda$, i.e., $x^\lambda:=(x^1,\ldots,x^{n-1},2\lambda-x^n)$. We also define $v_\lambda(x):=v(x^\lambda)$ and
$$w_\lambda(x):=v_\lambda(x)-v(x).$$
Note that the equation satisfied by $v_\lambda$ is the same as the satisfied by $v$. Although this fact is not clear for non-local operators, it is easily seen to be true in the Caffarelli-Silvestre extension \eqref{divergence-equation2}. Then, by linearity,
\be\label{difference-Alexandroff}P_\gamma w_\lambda= v_\lambda^{\frac{n+2\gamma}{n-2\gamma}}-v^{\frac{n+2\gamma}{n-2\gamma}}, \mbox{ weakly}.\ee

We will need a couple of preliminary results:

\begin{lemma}\label{lemma:start}
Let  $v$ be any function with asymptotics
\be\label{behavior-infinity} v(x)={|x|}^{2\gamma-n}\lp a+\sum_{i=1}^n \frac{b_i x^i}{|x|^2}+O\lp|x|^2\rp\rp \quad\mbox{when } |x|\to\infty,\ee
for some $a>0$. Then there exists $\lambda_0>0$ such that for all $\lambda\geq\lambda_0$,
$$w_\lambda(x)>0\quad\mbox{for all } x\in\Sigma_\lambda.$$
\end{lemma}

\begin{proof}
This is just Lemma 2.2. in \cite{Gidas-Ni-Nirenberg}, and it does not use \eqref{equation-euclidean}.
\end{proof}

\begin{lemma}\label{lemma:claim}
Let $v$ a weak solution of \eqref{equation-euclidean}.
If for some $\lambda<\lambda_0$ we have that $w_\lambda(x)\geq 0$ but $w_\lambda\not\equiv 0$ in $\Sigma_\lambda$, then
\be\label{properties}w_\lambda(x)>0 \mbox{ in }\Sigma_\lambda\quad \mbox{ and }\quad\partial_n v(x)<0 \mbox{ on }S_\lambda.\ee
\end{lemma}

\begin{proof}
When $v$ solves the constant scalar curvature equation, this is just Lemma 2.2. and Lemma 4.3 in
\cite{Gidas-Ni-Nirenberg:symmetry-max-principle}. In our case, we need a strong maximum principle and Hopf's lemma
for the operator $P_\gamma$ (see \cite{fractional-Yamabe}, \cite{Cabre-Sire:I}). We know from \eqref{difference-Alexandroff}
that $P_\gamma w_\lambda\geq 0$. Since $w_\lambda\geq 0$ in $\Sigma_\lambda$ (and is not identically zero), and
since $w_\lambda$ vanishes on the boundary $S_\lambda$, the strong maximum principle gives that $w_\lambda>0$
in all of $\Sigma_\lambda$. On the other hand, Hopf's lemma implies that $\partial_n w_\lambda>0$ on $S_\lambda$.
Then, $\partial_n w_\lambda=\partial_n v_\lambda-\partial_n v=-2\partial_n v$, so we immediately have $\partial_n v<0$
along $S_\lambda$.\\
\end{proof}

\noindent\emph{Proof of Theorem \ref{thm:convex-boundary}: } Let $g$ be a complete metric of constant positive scalar
$Q_\gamma$ curvature on $\Omega\subset S^n$ of the form $g=u^{\frac{4}{n-2\gamma}}g_c$, and $B$ an open ball in $\Omega$.
Let $S=\partial B$ be the boundary sphere. Fix any point $p \in S$. Use stereographic projection to map $\Omega$ into
$\tilde\Omega\subset\mathbb R^n$ so that $p$ is mapped to infinity. Then $S$ is transformed to a hyperplane $\tilde S$,
and the projected $\partial\tilde\Omega$ lies on one side of  $\tilde S$, say below. Use linear coordinates
$(x^1,\ldots,x^n) \in \mathbb R^n$ with $\tilde S = \{x^n=0\}$.

By stereographic projection, the metric $g$ transforms to a conformally flat metric on $\mathbb R^n$,
$g_v=v^{\frac{4}{n-2\gamma}}\abs{dx}^2$. Since the scalar curvature equation is conformally covariant, we also  have
$$
\Delta^\gamma v= v^{\frac{n+2\gamma}{n-2\gamma}},
$$
on $\tilde\Omega\subset  \mathbb R^n$.  Note that $v$ is a weak solution of this equation on all of $\RR^n$.

Since the function $u$ is smooth and strictly positive at $p$, the function $v$ is regular at infinity, i.e.\ has the asymptotics
\eqref{behavior-infinity} for some $a>0$ as $|x|\to\infty$.

\emph{Step 1. Starting the reflection:}
Thanks to Lemma \ref{lemma:start}, we can initiate the reflection argument when $\lambda$ is sufficiently large.
Note that the equation satisfied by $v$ is not needed here since we have the precise behavior \eqref{behavior-infinity}.

\emph{Step 2. Continuation:} We now move the plane $S_\lambda$, so long as it does not touch the singular set.
Suppose that at some $\lambda_1>0$ we have $w_{\lambda_1}(x)> 0$ for all $x\in\Sigma_{\lambda_1}$, but $w_{\lambda_1}\not\equiv 0$
in $\Sigma_{\lambda_1}$.  Then the plane can be moved further; more precisely, there exists some $\epsilon>0$ not depending
on $\lambda_1$ such that $w_\lambda\geq 0$ in $\Sigma_\lambda$ for all $\lambda\in[\lambda_1-\epsilon,\lambda_1]$.

We observe first that because of Lemma \ref{lemma:claim} we must have
\be\label{moving0}
\partial_n v<0 \quad \mbox{on }\Sigma_{\lambda_1}.
\ee
Next, the proof of our claim follows as in Lemma 2.3. in \cite{Gidas-Ni-Nirenberg} by contradiction. Thus, assume that
there is a sequence $\lambda_j\to \lambda_1$ and a sequence of points $\{x_j\}$, $x_j\in\Sigma_{\lambda_j}$ such that
\be\label{moving1}w_{\lambda_j}(x_j)\leq 0.\ee
Either a subsequence, which we again call $\{x_j\}$, converges to $x_\infty\in\Sigma_{\lambda_1}$ or else $x_j\to\infty$.
In the first case, because of \eqref{moving1} we must have $\partial_n v(x_\infty)\leq 0$, thus contradicting \eqref{moving0}.
So $x_j\to\infty$. But in this second case may use the asymptotics for $v$ from \eqref{behavior-infinity}, that imply
$$
\frac{|x_j|^n}{\lambda_j-x_j^n} w_{\lambda_j}(x) \to -(n-2\gamma) a<0.
$$
This is a contradiction to \eqref{moving1}.

Finally, note that in this process we never have  $w_\lambda \equiv 0$ since the existence of the singularity of $v$
implies that it has no plane of symmetry. Hence the moving plane can be moved all the way to $\lambda=0$.

\emph{Step 3. Conclusions:} We have shown that $\Sigma_\lambda$ can be moved to $\lambda = 0$,  and then
$w_\lambda(x)>0$ for all $x\in\Sigma_0$ and
$$
\partial_n v(x)<0\quad\mbox{for all }x\in\overline{\Sigma_0}.
$$
Since $g=v^{\frac{4}{n-2\gamma}}|dx|^2$, the second fundamental form of any plane $S_\lambda$, $\lambda\geq 0$,
with respect to $g$ is given by
$$
\lp -\frac{4}{n-2\gamma} v^{-1}\frac{\partial v}{\partial x^n} I\rp.$$
The sign of $\partial_n v$ therefore implies that $S_\lambda$ is locally geodesically convex for all $\lambda\geq 0$.
When transferred back to the sphere, this shows that any round ball contained in $\Omega$ has locally geodesically
convex boundary, as claimed.
\qed

\medskip

\textbf{Acknowledgements: }
M.d.M. Gonz\'alez is supported by Spain Government project MTM2008-06349-C03-01 and GenCat
2009SGR345. R. Mazzeo is supported by the NSF grant DMS-0805529.
Y. Sire would like to thank the hospitality of the Department of Mathematics at Stanford University.

% ==================================================================================

\bibliographystyle{abbrv}

\end{document}